\documentclass[12pt]{article} 
\usepackage{amssymb}
\usepackage{amsthm}
\usepackage{latexsym}
\newcommand{\R}{I\!\!R}

\def\bgneqy{\begin{eqnarray}}
\def\endeqy{\end{eqnarray}}
\def\bgneqy*{\begin{eqnarray*}}
\def\endeqy*{\end{eqnarray*}}

\newtheorem{thm}{Theorem}[section]

\newtheorem{Lem}[thm]{\normalsize{\bf{Lemma}}}

\newtheorem{Rem}[thm]{Remark}

\begin{document}
\title{Asymptotic expansions of the largest eigenvalues}

\author{Maroua Gozzi\thanks{Department of Mathematics, Faculty of Sciences, 7021 Zarzouna, Bizerte, Tunisia.
(Email:gozzi.maroua@yahoo.fr)} \and \ Abdessatar Khelifi \thanks{ Faculty of
Sciences of Bizerte, University of Carthage
(Email: abdessatar.khelifi@fsb.rnu.tn) }}

\date{}
\maketitle
\begin{abstract}
In this paper, we provide a rigorous derivation of asymptotic formula
for the largest eigenvalues using the convergence estimation of the eigenvalues
of a sequence of self-adjoint compact operators of perturbations resulting
from the presence of small inhomogeneities.
\end{abstract}
%
%

 \bigskip

\noindent {\footnotesize Mathematics Subject Classification
(MSC2010): 35R30, 35B30}
 \section{Introduction}
 In the past decades, the Laplacian spectrum has received much more and more attention, since it has been applied to several fields, such a  randomized algorithms, combinational optimization problem and machine learning. One of this attention, the so-called largest eigenvalue $\cite{Ern-Guermond}$, $\cite{Gaussier-Yvon}$ plays an important role in many techniques of multivariate statistics, including the principal Component Analysis, and possibility of their using in statical tests as test statistics $\cite{Tracy-C-Widom-H}$. For example, the study of sample covariance matrices in fundamental multivariate analysis. In this setting, relatively little is known about the distribution of the largest eigenvalue $\cite{Johnstone-I-M}$, $\cite{Bilodeau}$, $\cite{Soshnikov}$.
 The present paper give an asymptotic expansion for the largest eigenvalues using the convergence estimation of the eigenvalues
of a sequence of self-adjoint compact operators in perforated domain which are deeply based on the $\textsf{Osborn}'s $ formula in \cite{Osborn}.\\
 Let $\Omega$ be a bounded domain in ${\R}^{2}$, with lipchitz boundary $\partial{\Omega}$.
Let $\nu$ denote the out unit normal to $\partial\Omega$ and assume it has a smooth background conductivity 1. We suppose that $\Omega$ contains a finite number of small inhomogeneities each of the form $z_{l}+\epsilon B_{l}$, where $B_{l}\subset {\R}^{2}$ is a bounded smooth ($C^{\infty}$) domain containing the origin. The total collection of imperfections thus takes the form $D_{\epsilon}=\bigcup_{l=1}^N D_{\epsilon}^{l}$, where $ D_{\epsilon}^{l}=z_{l}+\epsilon B_{l}$. The points $z_{l}\in\Omega$, $l=1, 2, \ldots, N$, that determines the locations of the inhomogeneities.
Let $\lambda^{i}$ be the ith eigenvalue of multiplicity $\textit{m}$ for the Laplacian
in the absence of any inhomogeneities. Then there exist $\textit{m}$ nonzero solutions
$\{u^{ij}\}_{j=1}^m$ to
\begin{equation}
\left\{ \begin{array}{ll}
-\Delta u^{ij}=\lambda^{i} u^{ij}& \textrm{in $\Omega,$}\\
\\
\frac{\partial u^{ij}}{\partial\nu}|_{\partial \Omega}=0,\quad \int_{\Omega} |u^{ij}|^{2}=1.
\end{array} \right.
\end{equation}

The eigenvalues problem in the presence of imperfections consists of finding $\{ \lambda_{\epsilon}^{ij}\}_{j=1}^m$ such that there exists a nontrivial eigenfunctions $\{ u_{\epsilon}^{ij}\}_{j=1}^m$ that is solution to

\begin{equation}
\left\{ \begin{array}{ll}
-\nabla. (1+ \sum_{l=1}^{N}((k_{l}-1){\chi}(D_{\epsilon}^{l})))\nabla u_{\epsilon}^{ij}=\lambda_{\epsilon}^{ij}u_{\epsilon}^{ij}& \textrm{in $\Omega,$}\\
\\
\frac{\partial u^{ij}_{\epsilon}}{\partial\nu}|_{\partial \Omega}=0,\quad \int_{\Omega} |u_{\epsilon}^{ij}|^{2}=1.
\end{array} \right.
\end{equation}
\\
It is well known that all eigenvalues of (1) are real, of finite multiplicity, have no finite accumulation points and there corresponding eigenfunctions which make up an orthonormal basis of $ L^{2}(\Omega)$.\\
This paper is organized as follows. In section 2, we give some preliminaries results. In section 3, we derive the asymptotic expansion for the Largest eigenvalue using the convergence estimate of the eigenvalues of a sequence of self-adjoint compact operators and applying the $\textsf{Osborn}'s $ formula.\\
\section{Some preliminaries results}
To derive the asymptotic formula for the eigenvalues we will use a convergence estimate of the eigenvalues of a sequence of self-adjoint compact operators. Let X be a (real) Hilbert space and suppose we have a compact, self-adjoint linear operator $T :X\rightarrow X$ along with a sequence of compact, self-adjoint linear operators $T_{\epsilon}:X\rightarrow X $ such that $T_{\epsilon}\rightarrow T$ pointwise as $\epsilon\rightarrow 0$ and the sequence $\{ T_{\epsilon}\}$ is collectively compact.
Let $\mu$ be a nonzero eigenvalue of $T$ of multiplicity $\textit{m}$.
Then we know that for small $\epsilon$, each $T_{\epsilon}$ has a set of eigenvalues counted according to multiplicity, $\{ \mu_{\epsilon}^{1}, \ldots,\mu_{\epsilon}^{m}\}$ such that for each $j$, $\mu_{\epsilon}^{j}\rightarrow \mu$ as $\epsilon \rightarrow 0$.
Define the average
\begin{equation}
{\overline{\mu}_{\epsilon}}=\frac{1}{m}\sum_{j=1}^{m}\frac{1}{\mu_{\epsilon}^{j}}.
\end{equation}
 If $ \phi^{1}, \ldots, \phi^{m}$ is an orthonormal basis of eigenfunctions associated with the eigenvalue $\mu$, then there exists a constant C such that for $j= 1, \ldots, m$ the following ${{\textsf{Osborn}'s}}$ formula holds

\begin{equation}
|\mu-{\bar{\mu}_{\epsilon}}-\frac{1}{m}\sum_{j=1}^{m}<(T-T_{\epsilon})\phi^{j}, \phi^{j}>|\leq C\|(T-T_{\epsilon})|_{span\{\phi^{j}\}_{1\leq j \leq m}}\|^{2},
\end{equation}
where $(T-T_{\epsilon})|_{{span{\{\phi^{j}\}}}_{1\leq j \leq m}}$ denotes the restriction of $(T-T_{\epsilon})$ to the $\textit{m}$-dimensional vector space spanned by ${\phi^{j}}_{1\leq j \leq m}$.\\
\\
In our case, let X be $ L^{2}(\Omega)$ with the standard inner product. For any $g \in L^{2}(\Omega)$, we define $T_{\epsilon}g=u_{\epsilon}$ and $Tg = u$, where $u_{\epsilon}$ is the solution to

\begin{equation}
\left\{ \begin{array}{ll}
-\nabla. (1+ \sum_{l=1}^{N}((k_{l}-1)\chi(D_{\epsilon}^{l}))\nabla u_{\epsilon}= g & \textrm{in $\Omega$},\\
\\
\frac{\partial u_{\epsilon}}{\partial\nu}|_{\partial \Omega}=0, \quad \int_{\Omega} u_{\epsilon}=0,\\
\end{array} \right.
\end{equation}

and u is the solution of

\begin{equation}
\left\{ \begin{array}{ll}
\Delta u= g & \textrm{in $\Omega,$}\\
\\
{\frac{\partial u}{\partial\nu}}|_{\partial \Omega}=0, \quad \int_{\Omega}u=0.\\
\end{array} \right.
\end{equation}
The function $g\mapsto(-\Delta)^{-1}g$ is continuous from $L^{2}(\Omega)$ to $H^{1}_{0}(\Omega)$. Clearly $T_{\epsilon}$ and T are compact operators   from $L^{2}(\Omega)$ to $L^{2}(\Omega)$. From the standard $H^{1}$ estimates we get the following lemma:
\begin{Lem}
$T_{\epsilon}$ and T are compact, self-adjoint operators from $L^{2}(\Omega)$ to $L^{2}(\Omega)$.
Moreover, the family of operators $\{T_{\epsilon}\}$ is collectively compact.
\end{Lem}
Let $(\mu^{i},u^{i})$, and $(\mu_{\epsilon}^{i}, u_{\epsilon}^{i})$ be the ith normalised eigenpairs of T and $T_{\epsilon}$  respectively.
Then if $\lambda^{i}=\frac{1}{\mu^{i}} $ and $\lambda^{i}_{\epsilon}=\frac{1}{\mu^{i}_{\epsilon}} $, then $u_{\epsilon}^{i}$ and $u^{i}$ are the solution of (2) and (1) respectively. From the spectral theory, if $\lambda^{i}$ has a multiplicity $\textit{m}$ with a correspond set of orthonormal eigenfunctions $\{u^{ij}\}_{j=1}^{m}$ then there exist $\textit{m}$ eigenvalues $\lambda_{\epsilon}^{ij}$ that satisfies the following lemma.

\begin{Lem}
Let $\Omega$ be a bounded domain in ${\R}^{2}$, and $\epsilon > 0$ a small number. Then the eigenvalues of Laplacian-Neumann operator satisfy the following expansion
\begin{equation}
\lambda^{i}- \lambda^{ij}_{\epsilon}= \theta(1),\quad \textrm{as tends to } 0,
\end{equation}
in the other word,
\begin{equation}
\lambda^{i}\lambda_{\epsilon}^{ij}= \lambda^{i^{2}}+\lambda^{i}\theta(1),\quad \textrm{as tends to } 0,
\end{equation}
for any $j = 1, 2, \ldots, \textit{m}$, such that $\theta(1)$ independent of $i$.
\end{Lem}
\begin{proof}
From $\cite{Ammari-Kang-Li}$, the proof is based in the following asymptotic expansion
\begin{eqnarray*}
\omega_{\epsilon}-\omega_{0} &=& \frac{1}{2 \sqrt{\pi}} \sum_{p=1}^{\infty} \frac{1}{p} \sum_{n=p}^{+\infty}\epsilon^{n} tr \int_{\partial V_{\delta_{0}}} B_{n,p}(\omega)d\omega \\
                                    \nonumber\\
&+&\frac{1}{2 \sqrt{\pi}}\sum_{n=1}^{\infty}\frac{(2\pi)^{n}}{n}\int_{\partial V_{\delta_{0}}}\left[\frac{1}{\ln(\eta\omega\epsilon)}
\times(D_{\Omega}^{\varepsilon})\left[N_{\Omega}^{\varepsilon}(.,z)\right](x)-\frac{\ln cap(\partial B)}{2 \pi}\right]^{n}d\omega,
\end{eqnarray*}
where
\begin{eqnarray*}
B_{n,p}(\omega)&=& (-1)^{p}\sum_{n_{i}}(A_{0}(\omega)+\ln(\omega\epsilon)B_{0}(\omega))^{-1}(A_{n_{1}}+\ln(\omega\epsilon)B_{0}(\omega))\\
                                                    \nonumber\\
&\times&\ldots(A_{0}(\omega)\ln(\omega\epsilon)B_{0}(\omega))^{-1}(A_{n_{1}}+\ln(\omega\epsilon)B_{n,p}(\omega))\omega^{n},                                                    \end{eqnarray*}
 then $\omega_{\epsilon}$ is the characteristic eigenvalue and $D_{\Omega}$ is the double layer potential. We obtain from this asymptotic expansion the following leading-order term of $\lambda^{i}- \lambda^{ij}_{\epsilon}$ in two dimensions as follows
$$\lambda^{i}- \lambda^{ij}_{\epsilon}=\frac{-2 \pi}{\ln(\epsilon \sqrt{\lambda^{i}})}|u^{ij}(z)|^{2}+\theta(\frac{1}{|\ln(\epsilon)|}),$$
this formula is exactly the one derived by $\textsf{Ozawa}$ in $\cite{Ozawa}$, see also $\textsf{Besson}$ $\cite{Besson}$, so as $\epsilon$ tends to 0 we get our desired result.
\end{proof}
\section{Asymptotic expansion of the largest eigenvalues}
In this part, we assume that the domains $D_{\epsilon}^{l}$, l = 1, 2, \ldots,$\textit{m}$, satisfy
\begin{equation}
0<d_{0}\leq|z_{l}-z_{k}| \quad \forall l\neq k, \quad dist(z_{l},\partial\Omega)\geq d_{0} \quad \forall l,
\end{equation}
which the following theorem holds
\begin{thm}
Suppose that $\Omega$ contains an inclusion as the form $D= z+\epsilon B$ which is far from the boundary. Then the solutions $\{u^{ij}\}_{j=1}^{m}$ to (1) satisfy
\begin{itemize}
              \item[1)] $\|u^{ij}\|_{L^{\infty}(D)} \leq C$, where C is independent of $\epsilon$ and $\lambda^{i}$.
              \item[2)] $\|\frac{\nabla u^{ij}}{\sqrt{\lambda^{i}}}\|_{L^{\infty}(D)}\leq C$, where C is independent of $\epsilon$ and $\lambda^{i}$.
             \item[3)] $\|\frac{\nabla^{2} u^{ij}}{\lambda^{i}}\|_{L^{\infty}(D)}\leq C$, where C is independent of $\epsilon$ and $\lambda^{i}$.

            \end{itemize}
\end{thm}
%
%
%
To prove the Theorem 3.1, we need this following lemma
\begin{Lem}
For any $i,s = 0, 1, 2, \ldots$, we have
\begin{equation}
u_{s}^{ij}(r, \phi) = \frac{J_{s}(\beta_{si}\frac{r}{R})}{R\sqrt{\pi\{1-\frac{s^{2}}{\beta_{si}^{2}}\}J_{s}^{2}(\beta_{si})}}e^{\pm Is\phi},
\end{equation}
where $I^{2}= -1$, $J_{s}(r)$ is the Bessel function of integer order s, $\beta_{si}$ denotes the ith zero of $J'_{s}(r)$ and $\lambda_{s}^{i}=(\frac{\beta_{si}}{R})^{2}$ is the associated eigenvalue.
\end{Lem}

\begin{proof}
For a round disk of radius R, the Helmholtz equation $(\Delta+ \lambda_{s}^{i})u^{ij}=0$ can be expressed in the polar coordinate system in the following form (we can see $\cite{Mezhlum-Antonio}$)
$$\frac{\partial^{2}u^{ij}}{\partial r^{2}}+ \frac{1}{r} \frac{\partial u^{ij}}{\partial }+\frac{1}{r^{2}}\frac{\partial^{2}u^{ij}}{\partial \theta^{2}}+\lambda_{s}^{i}u^{ij}=0,$$
with $0\leq r\leq R$, $0\leq \theta\leq 2\pi.$\\
\\
 One can seek a solution of the last equation as Fourrier expansion over $0\leq \theta\leq 2\pi$, that
$$u^{ij}=\sum_{-\infty}^{\infty}U_{si}(r)e^{is\theta}.$$
Then, for each $U_{si}(r)$, due to linear independence of trigonometric functions, we arrives at to ordinary differential equation
$$U''_{si}(r)+\frac{1}{r}U'_{si}(r)+(\lambda_{s}^{i}-\frac{s^{2}}{r^{2}})U_{si}(r)=0,$$
whose only solution regular inside the disk is the Bessel function of the first kind and order $m$ (see Abramowitz and Stegum, 1965).
Therefore a complete system of linearly independent solutions for our equation, can be chosen as $\{J_{s}\times e^{(\pm I s \phi)}\}$, with the boundary Neumann condition $J'_{s}=0$.\\
By a separation of the variables we can  write $u_{s}^{ij}(r, \phi)= c_{si}U_{si}(r)e^{\pm I s \phi}$, where $c_{si}$ is a constant and $U_{si}(r)$ satisfies
\begin{displaymath}
\left\{ \begin{array}{ll}
U''_{si}(r)+\frac{1}{r}U'_{si}(r)+(\lambda_{s}^{i}-\frac{s^{2}}{r^{2}})U_{si}(r)=0, &  \textrm{$0\leq r < R,$}\\
\\
U'_{si}(R)=0.
 \end{array} \right.
\end{displaymath}
Using the definition of the Bessel function $\cite{Ammari}$ we can deduce that,
$$U_{si}(r, \phi)=J_{s}(\beta_{si}\frac{r}{R}),$$
we get,

$$u_{s}^{ij}(r, \phi) = \frac{J_{s}(\beta_{si}\frac{r}{R})}{R\sqrt{\pi\{1-\frac{s^{2}}{\beta_{si}^{2}}\}J_{s}^{2}(\beta_{si})}}e^{\pm Is\phi}.$$

\end{proof}
Now, we give an important result about the eigenvalue $\lambda^{i}$, which will be described by this lemma
\begin{Lem}
As i tends to $\infty$, we have:
\begin{equation}
\lambda^{i}\approx i.
\end{equation}
\end{Lem}

\begin{proof}
We proof our lemma  firstly in the case of a disk and after we study the general case.
\begin{itemize}
  \item The case of a disk:
\end{itemize}
Our proof are based of this relation  $\lambda_{s}^{i}=(\frac{\beta_{si}}{R})^{2}$.
From, $\cite{Vlademir}$, if we fix s thus $\beta_{si}$ have the following asymptotic expansion

$$\beta_{si}=\beta'_{si}+ \theta(\frac{1}{\beta'_{si}}),$$

where $ \beta'_{si} = (i+\frac{s}{2}-\frac{3}{4})\pi$.\\
\\
Then, we have
\begin{itemize}
  \item The general case:
\end{itemize}
Here we use the $\textsf{Weyl}'s$ Asymptotic Formula in $\cite{Wolfgang-Robin-Wolfgang-Frank}$ to get the asymptotic expansion for the eigenvalues. Let X be a Riemannian manifold and $\Delta$ be the Laplace operator on X. It is well-known that if X is compact then the spectrum of $-\Delta$ is discrete and consists of an increasing sequence $\{ \lambda^{i}\}_{i=1}^{\infty}$ of the eigenvalues (counted according their multiplicities) where $\lambda^{1}= 0$ and $\lambda^{i}\rightarrow \infty$ as $i\rightarrow \infty$. Moreover, if $n= dim X$ then $\textsf{Weyl}'s$ asymptotic formula says that
\begin{equation}
\lambda^{i}\sim c_{n} (\frac{i}{\mu(X)})^{\frac{2}{n}},\quad  i\rightarrow \infty,
\end{equation}
where $\mu$ is the Reimannian measure on X and $ c_{n}>0$ is a constant depending only on $n$. According the last theory in our case when, $i$ is large, $\lambda ^{1}=0$ corresponding to $ u^{1}= \frac{1}{\sqrt{|\Omega|}}$, $ X= \Omega$ and $ \mu(\Omega)= |\Omega|$ we can prove the above lemma.
\end{proof}
Now, we prove Theorem 3.1\\
\\
$\textbf{Proof of Theorem 3.1}.$
We derive this theorem firstly in the simple case when $\Omega$ is a disk and after we study the general case;
\begin{itemize}
  \item The case of a disk:
\end{itemize}
The eigenvalue $\{\lambda_{s}^{i}\}_{i, s= 0, 1, 2, ...}$ of $-\Delta$ in a disk $\Omega$ of radius R in ${\R}^{2}$ have two multiplicity and they are the solutions of the following system:

\begin{equation}
\left\{ \begin{array}{ll}
(\Delta+ \lambda_{s}^{i})u^{ij}_{s}(r,\phi)& \textrm{in $\Omega,$}\\
\\
\frac{\partial u^{ij}_{s}(r,\phi)}{\partial r}|_{r=R}=0, & \textrm{$\int_{\Omega} |u^{ij}_{s}|^{2}=1,\quad  j=1,2.$}
\end{array} \right.
\end{equation}
From the Lemma 3.1, we have the following result
\begin{equation}
u_{s}^{ij}(r, \phi) = \frac{J_{s}(\beta_{si}\frac{r}{R})}{R\sqrt{\pi\{1-\frac{s^{2}}{\beta_{si}^{2}}\}J_{s}^{2}(\beta_{si})}}e^{\pm Is\phi}.
\end{equation}
Then, we study the eigenfunction
\begin{enumerate}
  \item  As $r\rightarrow \infty$ we have $J_ {s}(r)\simeq \frac{s!}{2^{s}}r^{s}$.
  \item  As $\beta \rightarrow \infty$ we have $ J_{s}(\beta)\simeq \sqrt{\frac{2}{\pi\beta}}\cos (\beta-\frac{(2S+1)}{2}\frac{\pi}{4})+\theta (\frac{1}{\beta^{\frac{3}{2}}})$.
  \item If we fixe s and we choose $i= E(\frac{1}{\epsilon^{\alpha}})$, we find

$$\beta_{si}\approx \frac{1}{\epsilon^{\alpha}}+ \theta(\epsilon^{-\alpha}).$$
\end{enumerate}
Using, this three assertions we can deduce that $| \frac{J_{s}(\beta_{si}\frac{r}{R})}{J_{s}(\beta_{si})}|$ is uniform bound for $s$ and $r\in [0, R],$ so $\|u^{ij}\|_{L^{\infty}(D)} \leq C$, where C is independent of $\epsilon$ and $\lambda^{i},$ where satisfied for the case of a disk.
Now, let's turn to the general case;
\begin{itemize}
  \item The general case:
\end{itemize}
From [4], $u^{ij}$ can be represented as
$$u^{ij}=S^{\sqrt{\lambda^{i}}}_{D}\varphi(x), \quad x\in D, $$
where $S^{\sqrt{\lambda^{i}}}_{D}\varphi(x)$ the single layer potential of the density function $\varphi$ on $\partial D$, which can be defined as
$$S^{\sqrt{\lambda^{i}}}_{D}\varphi(x)= \int_{\partial D}\Gamma(x-y)\varphi(y)d\sigma(y),\quad x\in {\R}^{2}.$$
On the other hand,  we have the following result,
let $\varphi\in L^{2}(\partial D)$, $\tilde{\varphi}(x)=\epsilon\varphi(\epsilon x+z), x\in \partial B$. Then, for $x \in \partial B$, we have
\begin{eqnarray*}
S^{\sqrt{\lambda^{i}}}_{D}\varphi(\epsilon x+z)&=& \frac{1}{2\pi}\sum_{n=0}^{+\infty}(-1)^{n} \frac{(\sqrt{\lambda^{i}}\epsilon)^{2n}}{2^{2n}(n!)^{2}}\\
                                                                        \nonumber\\
&\times& \int_{\partial B}|x-y|^{2n}(\ln(\sqrt{\lambda^{i}}\epsilon |x-y|)+\ln\gamma-\sum_{j=1}^{n}\frac{1}{j})\tilde{\varphi}(y)d\sigma(y).                                                                      \end{eqnarray*}
Then, using the Lemma 3.2 we can get the assertion (1) deduced easily such that,
$$\|u^{ij}\|_{L^{\infty}(D)}=\{\sup |u^{ij}|, \quad x\in D \}.$$
Finally, the last two assertions may be deduced easily after we calculate $\frac{\nabla u^{ij}}{\sqrt{\lambda^{i}}}$ and $\frac{\nabla^{2} u^{ij}}{\lambda^{i}}.\square$
\subsection{Estimation energy}
In this section, when the inclusions are not degenerate (i.e. their conductivity $ k_{l}> 0$, $k_{l}\neq 1$) the first term the expansion of $u_{\epsilon}$ solution to (5) is the background potential u solution to (6). In fact, $u_{\epsilon}$ converges strongly in $H^{1}(\Omega)$. This is the consequence of the following estimate of the $H^{1}(\Omega)$ norm of $ u_{\epsilon}-u$.
\begin{Lem}
Let $u_{\epsilon}$ be the solution to (5) and u solution to (6) for a given $g\in L^{2}(\Omega)$. Then there exists a constant C, independent of $\epsilon$, u and the set of points $(z_{l})_{l=1}^{N}$ such that the following estimate holds:
\begin{equation}
\|u_{\epsilon}-u\|_{H^{1}(\Omega)}\leq C(\|\nabla_{x}u\|_{L^{\infty}(D_{\epsilon})}\epsilon^{\frac{3}{2}}+\|\nabla^{2}_{x}u\|_{L^{\infty}(D_{\epsilon})}\epsilon^{2}
+\|g\|_{L^{\infty}(D_{\epsilon})}\epsilon^{2}).
\end{equation}

\end{Lem}
\begin{proof}
The proof of the above lemma is based in the following estimate
\begin{equation}
\|\nabla_{y}(u_{\epsilon}(\epsilon y)-u(\epsilon y)-\epsilon \upsilon(y))\|_{L^{2}({\tilde{\Omega}})}\leq C(\|\nabla_{x}u\|_{L^{\infty}(D_{\epsilon})}\epsilon^{\frac{3}{2}}+\|\nabla^{2}_{x}u\|_{L^{\infty}(D_{\epsilon})}\epsilon^{2}
+\|g\|_{L^{\infty}(D_{\epsilon})}\epsilon^{2}).
\end{equation}

where $\upsilon$ is the unique solution of the following transmission problem:
\begin{equation}
\left\{ \begin{array}{ll}
 \Delta\upsilon=0& \textrm{in $ {\R}^{2}\backslash(\overline{\bigcup_{l=1}^{m}B_{l}})$}\\
 \\
\Delta\upsilon=0& \textrm{ in $B_{l}, \quad \forall l= 1,2, 3, \ldots, m.$}\\
\\
\upsilon|_{-}=\upsilon|_{+}& \textrm{ on $\partial B_{l}$}\\
\\
\frac{\partial \upsilon}{\partial\nu}|_{+}-k\frac{\partial \upsilon}{\partial\nu}|_{-}=(k-1)\nabla_{x}u(z_{l}).\nu_{l}, & \textrm{ on $\partial B_{l},$}\\
\\
\lim_{|y|\rightarrow\infty}\upsilon(y)=0,
 \end{array} \right.
\end{equation}
and ${\tilde{\Omega}}= \frac{1}{\epsilon}\Omega$. Then, from $\cite{Khelifi}$ we put
$$\omega(\epsilon)=u_{\epsilon}(\epsilon y)\epsilon-u(\epsilon y)-\epsilon \upsilon(y).$$
Using the unperturbed problem and setting $\lambda=\lambda_{j}(\epsilon)$, we see that $\omega(\epsilon)$ solves:
$$-\Delta \omega(\epsilon)=\lambda \omega+(\lambda-\lambda_{0})(u(\epsilon y)+\epsilon \upsilon(y)).$$
For $z \in {\R}$, we define the function $\theta$ by
$$\theta(z)=\lambda z + (\lambda-\lambda_{0})\|u(\epsilon y)-\epsilon \upsilon(y)\|_{{L^{\infty}}(D_{\epsilon})},$$
then trivially remark,

$$|\theta(z)|\leq |\theta(0)|+\lambda |z|, \quad \forall z\in {\R},$$

and consequently,
\begin{equation}
|\theta(\omega(z))|\leq|\theta(0)|+\lambda|z|.
\end{equation}
Now, it turns out from the definition of $ \omega $ that $\omega(\epsilon)\rightarrow 0$ as $\epsilon\rightarrow 0$ and so by
$\|u(\epsilon  y)-\epsilon \upsilon(y)\|_{L^{\infty}_{(D_{\epsilon})}}>0$, we get,
$$|\omega(\epsilon)(x)|\leq 2 \|u(\epsilon  y)+\epsilon \upsilon(y)\|_{L^{\infty}_{(D_{\epsilon})}},\quad \textrm{for }x\in D_{\epsilon}.$$
Moreover, we recall that $\lambda(\epsilon)\rightarrow\lambda_{0}$.
Now it is useful to introduce the following function
$$\tilde{\theta}(\omega(\epsilon))=\lambda \omega +(\lambda-\lambda_{0})(u(\epsilon y)+\epsilon \upsilon(y)),$$
and the second term is bounded by
$$\lambda |\omega(\epsilon)|\leq 2\lambda_{0} \|u(\epsilon  y)+\epsilon \upsilon(y)\|_{L^{\infty}_{(D_{\epsilon})}}.$$
These estimate give
\begin{equation}
\| \tilde{\theta}(\omega(\epsilon))\|_{L^{\infty}({\Omega})}\leq (1+2\lambda)\|u(\epsilon  y)+\epsilon\upsilon(y)\|_{L^{\infty}_{(D_{\epsilon})}}.
\end{equation}
Next, we can write in $ {\Omega}$,
$$-\Delta\omega(\epsilon)=\tilde{\theta}(\omega(\epsilon)).$$
By integrating by parts, we find that the function $\omega(\epsilon)$
is a solution to the following problem
$$\forall v \in H^{1}_{0}({\Omega}),\quad  \int_{{\Omega}}\nabla\omega(\epsilon){\nabla}v dx=\int_{{\Omega}}{v}(\omega(\epsilon)){v}dx.$$
If we take $v=\omega(\epsilon)$ we can deduce that:
$$\|\nabla \omega\|^{2}_{L^{2}({\Omega})}=\int_{{\Omega}}{\Omega}(\omega)\bar{\omega}dx,$$
Then, by Poincare's inequality, there exist some positive constant $C(\tilde{\Omega})$ such that
\begin{equation}
\|\omega\|_{L^{2}({\Omega})}\leq C({\Omega})\|\nabla \omega\|_{L^{2}({\Omega})}.
\end{equation}
Since $\omega$ and $\nabla\omega$ are uniformly bounded in ${\Omega}$. There exist some constant independent of $\epsilon$ such that $C({\Omega})\leq C_{0}$, which give
\begin{equation}
\|\nabla\omega\|_{L^{2}({\Omega})}\leq C \|u(\epsilon y)+\epsilon\upsilon(y)\|_{L^{\infty}_{(D_{\epsilon})}},
\end{equation}
which concludes the proof.
\end{proof}
Now, we have this remark
 \begin{Rem}
The function $v$ is connected to polarisation tensors ${M}^{(L)}$ for any $l= 1, \ldots, N$, which are given by
\begin{equation}
{M}^{(L)}_{pq}=(1-k_{l})|B_{l}|\delta_{pq}+(1-k_{l})^{2}\int_{\partial B_{l}}y_{p} \frac{\partial\phi_{p}^{(l)}}{\partial \nu}|_{-}d\sigma_{y},
\end{equation}

where for $ p=1, 2$, $\phi_{p}^{(l)}$ is the unique function which satisfies

\begin{equation}
\left\{ \begin{array}{ll}
\Delta \phi_{p}^{(l)}=0& \textrm{ in ${\R}^{2}\backslash\overline{B_{l}}$}\\
\\
\Delta \phi_{p}^{(l)} =0 & \textrm{ in $B_{l},$}\\
\\
\frac{\partial\upsilon}{\partial\upsilon}|_{+}-k\frac{\partial \phi_{p}^{(l)}}{\partial\upsilon}|_{-}=\upsilon_{l}, & \textrm{ on $\partial B_{l},$}\\
\end{array} \right.
\end{equation}

with $\phi_{p}^{(l)}$ continuous across $\partial B_{l}$ and $\lim_{|y|\rightarrow\infty}\phi_{p}^{(l)}=0.$

 \subsection{Derivation of the asymptotic expansion for the largest eigenvalues}
In this section, we restrict to the case of a single inhomogeneity (N=1), by iteration, we can get the more general case. So, we suppose that this inhomogeneity is centred at the origin, so it is of the form $D=\epsilon B$, with conductivity k. The general case may be verified by fairly direct iteration of the argument we present here, adding one inhomogeneity at a time. By according $\textsf{Osborn}'s $ formula in (4), we obtain
\begin{equation}
|\frac{1}{\lambda_{i}}-\frac{1}{m}\sum_{j=1}^{m}\frac{1}{\lambda^{ij}_{\epsilon}} -\frac{1}{m}\sum_{j=1}^{m}< \frac{1}{\lambda_{i}}u^{ij}-\upsilon_{\epsilon}^{ij}, u^{ij}>|\leq C \| \frac{1}{\lambda_{i}}u^{ij}-\upsilon_{\epsilon}^{ij}\|_{L^{2}(\Omega)}^{2},
\end{equation}

where $\upsilon_{\epsilon}^{ij}$ satisfies
\begin{displaymath}
\left\{ \begin{array}{ll}
\nabla.(1+ \sum_{l=1}^{m}(k_{l}-1){\chi}(D_{\epsilon}^{l})\nabla\upsilon_{\epsilon}^{ij})=u^{ij}& \textrm{ in $\Omega$,}\\
\\
\frac{\partial \upsilon_{\epsilon}^{ij}}{\partial\nu}|_{\partial\Omega}=0, \quad \int_{\partial\Omega}\upsilon_{\epsilon}^{ij}=0. \\
\end{array} \right.
\end{displaymath}

\end{Rem}

 If we take $i= (\frac{1}{\epsilon^{\alpha}})$ with  $0\leq \alpha\leq1$, we deduce from Lemma 3.5,
\begin{equation}
\lambda^{i}=\theta(\epsilon^{-\alpha}),\quad  as\quad  \epsilon \quad \textrm{tends to } 0.
\end{equation}

Then according to Theorem 3.1 and Lemma 3.7 with $u=\frac{u^{ij}}{\lambda^{i}}$, $u_{\epsilon}=\upsilon_{\epsilon}^{ij}$ and $g=u^{ij}$ we obtain
$$\|\frac{u^{ij}}{\lambda^{i}}-\upsilon_{\epsilon}^{ij}\|_{L^{2}(\Omega)}\leq C \frac{\epsilon^{\frac{3}{2}}}{\sqrt{\lambda^{i}}}.$$

This gives

\begin{equation}
\frac{1}{\lambda_{i}}-\frac{1}{m}\sum_{j=1}^{m}\frac{1}{\lambda^{ij}_{\epsilon}}= \frac{1}{m}\sum_{j=1}^{m}<\frac{1}{\lambda_{i}}u^{ij}-\upsilon_{\epsilon}^{ij},u^{ij}>+\theta(\frac{\epsilon^{3}}{\lambda^{i}}).
\end{equation}

 We integrate by parts and use the transmission conditions satisfies by $\upsilon_{\epsilon}^{ij}$ across $\partial D_{\epsilon}$, we get
 \begin{eqnarray*}
 <  \frac{1}{\lambda^{i}}u^{ij}-\upsilon_{\epsilon}^{ij},u^{ij} > & = &\int_{\Omega}(\frac{1}{\lambda^{i}} u^{ij}-\upsilon^{ij}_{\epsilon})u^{ij}dy
                                      \nonumber\\
 & = &\frac{k-1}{\lambda^{i}k}\int_{D}|u^{ij}|^{2}dy+ \frac{1-k}{\lambda^{i}}\int_{\partial D}\frac{\partial\upsilon_{\epsilon}^{ij}}{\partial\nu}|_{-}u^{ij}d\sigma_{x}.
\end{eqnarray*}

 Suppose
$$r_{\epsilon}^{ij}(x)=\upsilon_{\epsilon}^{ij}(x)-\frac{1}{\lambda^{i}}u^{ij}(x)-\epsilon\upsilon (\frac{x}{\epsilon}),$$

where $\upsilon$ is defined in (17) (with $\frac{1}{\lambda^{i}}u^{ij}$ in place of u), inserting this into the above formula we get
 \begin{eqnarray}
 \lefteqn{
<  \frac{1}{\lambda^{i}}u^{ij}-\upsilon_{\epsilon}^{ij},u^{ij} > =\frac{k-1}{\lambda^{i}k}\int_{D}|u^{ij}|^{2}dx+\frac{1-k}{\lambda^{i}}\int_{\partial D}\frac{\partial r_{\epsilon}^{ij}}{\partial\nu_{x}}|_{-}u^{ij}d\sigma_{x}}
                                             \nonumber\\
&&{}+ \frac{1-k}{\lambda^{i}}(\frac {1}{\lambda^{i}} \frac {\partial u^{ij}}{\partial\nu_{x}}(x)+\frac{\epsilon}{\lambda^{i}}\int_{\partial D}\frac{\partial\upsilon}{\partial\nu_{x}}|_{-}(\frac{x}{\epsilon}))u^{ij}d\sigma_{x}.
\end{eqnarray}

Note that
$$\Delta_{x}r_{\epsilon}^{ij}(x)=\frac{1}{k}u^{ij}(x)-\frac{1}{\lambda^{i}}\Delta_{x}u^{ij}(x).$$

From (12) we also have
$$\|\nabla_{\xi}r_{\epsilon}^{ij}(\epsilon\xi_{1}, \epsilon\xi_{2})\|_{L^{2}(\tilde{\Omega})}\leq C(k, B) \frac{\epsilon^{\frac{3}{2}}}{\sqrt{\lambda^{i}}},$$
this gives that
\begin{eqnarray*}
\lefteqn{ \frac{1-k}{\lambda^{i}}\int_{\partial D}\frac{\partial r_{\epsilon}^{ij}}{\partial\upsilon_{x}}|_{-}u^{ij}d\sigma_{x}=(1-k)\epsilon\int_{B}\nabla_{\xi}r_{\epsilon}^{ij}
(\epsilon\xi).\frac{\nabla_{x}u^{ij}(\epsilon\xi)}{\lambda^{i}}d\xi }
                                                   \nonumber\\
&&{}+(1-k) \epsilon^{2}\int_{B}(\frac{1}{k}u^{ij}(\epsilon\xi)-\frac{1}{\lambda^{i}}\Delta_{x}u^{ij}(\epsilon\xi))\frac{u^{ij}}{\lambda^{i}}(\epsilon\xi)d\xi.
\end{eqnarray*}

Using the Lemmas 3.5 and Theorem 3.1, we deduce that
\begin{eqnarray*}
|(1-k)\epsilon\int_{B}\nabla_{\xi}r^{ij}_{\epsilon}(\epsilon\xi).\frac{\nabla_{x}u^{ij}(\epsilon\xi)}{\lambda^{i}}d\xi| &\leq & C(k,B)\epsilon\|\nabla_{\xi}r^{ij}_{\epsilon}(\epsilon\xi)\|_{L^{2}(\tilde{\Omega})}\|\frac{\nabla_{x}u^{ij}}{\lambda^{i}}\|_{L^{\infty}(D)}
                                                    \nonumber\\
& \leq & C(k,B)\frac{\epsilon^{\frac{5}{2}}}{\sqrt{\lambda^{i}}},
\end{eqnarray*}

where C(k, B) independent of $i$.\\
\\
We conclude that
\begin{eqnarray}
 \frac{1-k}{\lambda^{i}}\int_{\partial D}\frac{\partial r_{\epsilon}^{ij}}{\partial \upsilon_{x}}|_{-}u^{ij}d\sigma_{x}& = &\frac{1-k}{k \lambda^{i}}\epsilon^{2}\int_{B}(u^{ij})^{1}(\epsilon\xi)d\xi-\frac{1-k}{\lambda^{{i}^{2}}}\epsilon^{2}
\int_{B}\Delta_{x}u^{ij}_{\epsilon}(\epsilon\xi)u^{ij}(\epsilon\xi)d\xi
                                                               \nonumber\\
& + &\theta(\frac{\epsilon^{\frac{5}{2}}}{\sqrt{\lambda^{i}}}).
\end{eqnarray}

 In the same time we have
\begin{eqnarray*}
(1-k)\int_{\partial D}\frac{1}{\lambda^{i}}\frac{\partial u^{ij}}{\partial\nu_{x}}(x)+\frac{\partial\upsilon}{\partial\nu_{\xi}}|_{-}(\frac{x}{\epsilon})
\frac{u^{ij}}{\lambda^{i}}d\sigma_{x}& = &
(1-k)\int_{\partial D}\frac{1}{\lambda^{i}} \frac{\partial u^{ij}}{\partial\nu}\frac{u^{ij}}{\lambda^{i}}d\sigma_{x}
                                                 \nonumber\\
& + & (1-k)\epsilon\int_{\partial D}\frac{\partial\upsilon}{\partial\nu_{\xi}}|_{-}(\frac{x}{\epsilon})\frac{u^{ij}}{\lambda^{i}}(x)d\sigma_{x}.
\end{eqnarray*}

 From Theorem 3.1, the Taylor expansion of $u^{ij}$ is

\begin{equation}
\frac{u^{ij}}{\lambda^{i}}(x)=\frac{u^{ij}}{\lambda^{i}}(0)+\frac{\nabla_{x}u^{ij}}{\lambda^{i}}(0). x+\theta(x^{2}),
\end{equation}
where $\theta(x^{2})$ independent of $i$. This gives that:
\\
\begin{eqnarray}
 \frac{1-k}{\lambda^{{i}^{2}}}\int_{\partial D}\frac{\partial u^{ij}}{\partial\nu_{x}}u^{ij}(x)dx&=& \frac{1-k}{\lambda^{{i}^{2}}}\int_{D}\Delta_{x}u^{ij}(x)u^{ij}(x)dx
                                              \nonumber\\
&+&(1-k)\int_{\partial D}\frac{|\nabla_{x}u^{ij}|^{2}}{\lambda^{{i}^{2}}}(x)dx
                                           \nonumber\\
&=&\epsilon^{2}\frac{1-k}{\lambda^{{i}^{2}}}\int_{B}\Delta_{x}u^{ij}(\epsilon\xi)u^{ij}(\epsilon\xi)dx
                                                    \nonumber\\
&+& \epsilon^{2}\frac{1-k}{\lambda^{{i}^{2}}}|B||\nabla_{x}u^{ij}(0)|^{2}+\theta(\frac{\epsilon^{3}}{\sqrt{\lambda^{i}}}),
\end{eqnarray}

 where, $\theta(\epsilon^{3})$ independent of $i$.\\
 \\
At the same time, by a Taylor expansion of $u^{ij}$ in (33) about $x=0$, we obtain

\begin{eqnarray}
 (1-k)\int_{\partial D}\frac{\partial \upsilon}{\partial \nu _{\xi}}|_{-}(\frac{x}{\epsilon})\frac{u^{ij}}{\lambda^{i}}(x)d\sigma_{x}
&=&(1-k)\int_{\partial D}\frac{\partial \upsilon}{\partial \nu_{\xi}}|_{-}(\frac{x}{\epsilon})\{\frac{u^{ij}}{\lambda^{i}}(0)+\frac{\nabla_{x}u^{ij}}{\lambda^{i}}(0).x+\theta(x^{2})\}d\sigma_{x}
                                                                                                        \nonumber\\
&=&\epsilon^{2}(1-k)\int_{\partial B}\frac{\partial\upsilon}{\partial\nu_{\xi}} |_{-}(\xi)\frac{\nabla_{x}u^{ij}}{\lambda^{i}}(0).\xi d\sigma_{\xi}+\theta(\frac{\epsilon^{3}}{\sqrt{\lambda^{i}}} ),
\end{eqnarray}

 where $\theta(\epsilon^{3})$ independent of $i$.\\
We choose $0\leq \alpha \leq 1$ and inserting the above identity, (29), (28) in (27), we get
\begin{eqnarray*}
\lefteqn{ <\frac{1}{\lambda^{i}}u^{ij}-\upsilon_{\epsilon}^{ij},u^{ij}>=\epsilon^{2}\frac{1-k}{\lambda^{{i}^{2}}}
(|B||\nabla_{x}u^{ij}(0)|^{2} }
                                           \nonumber\\
 & & {}+\lambda^{i}\int_{\partial B}\frac{\partial \upsilon}{\partial\nu_{\xi}}|_{-}(\xi)\frac{\nabla_{x}u^{ij}}{\lambda^{i}}(0).\xi d\sigma_{\xi})+\theta(\frac{\epsilon^{\frac{5}{2}}}{\sqrt{\lambda^{i}}}).
 \end{eqnarray*}

We now use the fact
\begin{equation}
\upsilon(\xi) = \frac{1-k}{\lambda^{i}}\sum_{p=1}^{2}\frac{\partial u^{ij}}{\partial x_{p}}(0)\phi_{p}(\xi),
\end{equation}
where $\phi_{p}$ is defined in (23). Putting $v^{ij}=\frac{u^{ij}}{\sqrt{\lambda^{i}}}$ we obtain the following expansion
$$\frac{1}{\lambda^{i}}-\frac{1}{m}\sum_{j=1}^{m}\frac{1}{\lambda^{ij}_{\epsilon}}=\frac{\epsilon^{2}}{m \lambda^{i}}\sum_{j=1}^{m} \nabla_{x}\upsilon^{ij}(0).{M}\nabla_{x}\upsilon^{ij}(0)+\theta(\frac{\epsilon^{\frac{5}{2}}}{\sqrt{\lambda^{i}}}),$$

where ${M}$ is the polarisation tensor which is defined in (22) and $\theta(\epsilon^{\frac{5}{2}})$  independent of $i$.\\

Our main result in this section is the following:

\begin{thm}

For any $i=E(\frac{1}{\epsilon^{\alpha}})$. If $\lambda^{i}$ is an eigenvalue of (1) of multiplicity $\textit{m}$ then there is $\textit{m}$ eigenvalue of (2) which converges to $\lambda^{i}$, such that
$$|\overline{\lambda^{i}_{\epsilon}}-\lambda^{i}|=\theta(\epsilon^{2-\alpha}),\quad \textrm{for any } 0\leq\alpha\leq 1,$$
for any small $\epsilon$.
\end{thm}

\begin{proof}
We easily see that far all $j=1,\ldots, m$ (see $\cite{Khelifi}$ for more details)
$$|\mu_{\epsilon}-\mu|=|\frac{\lambda^{i}_{\epsilon}-\lambda^{i}}{\lambda^{i}_{\epsilon}\lambda^{i}}|,$$
which we can write,
$$|\lambda^{i}_{\epsilon}-\lambda^{i}|=|\mu_{\epsilon}-\mu||\lambda^{i}_{\epsilon}\lambda^{i}|.$$
Then according to Lemma 2.2, we have
\begin{equation}
{\lambda^{i}_{\epsilon}}-\lambda^{i}=\epsilon^{2}(\lambda^{i}+\theta(1))\nabla_{x}\upsilon^{ij}(0).{M}\nabla_{x}\upsilon^{ij}(0)
+\theta(\epsilon^{\frac{5}{2}}(\lambda^{i}+\theta(\epsilon^{2-\alpha}))).
 \end{equation}
The fact that $$|\nabla_{x}\upsilon^{ij}(0).{M}\nabla_{x}\upsilon^{ij}(0)|\leq C,$$
where C is independent of $\lambda^{i}$ and $\epsilon$ which complete the proof.
\end{proof}
So, the following theorem holds.

\begin{thm}
Suppose $ \lambda^{E(\frac{1}{\epsilon^{\alpha}})}$ is an eigenvalue of multiplicity $m$ of (1), with an $L^{2}$ orthonormal basis of eigenfunctions $\{u^{E(\frac{1}{\epsilon^{\alpha}})j}\}_{j=1}^{m}$.
Suppose $ \lambda_{\epsilon}^{E(\frac{1}{\epsilon^{\alpha}})j}$ are eigenvalues of (2) which converge to $\lambda ^{E(\frac{1}{\epsilon^{\alpha}})}$. For any $0\leq \alpha \leq 1$, the following asymptotic expansion holds:
\begin{equation}
\overline{\lambda _{\epsilon}^{E(\frac{1}{\epsilon^{\alpha}})}}-\lambda^{E(\frac{1}{\epsilon^{\alpha}})}=
\epsilon^{2}\sum_{j=1}^{m}\sum_{l=1}^{N}\nabla_{x}u^{(\frac{1}{\epsilon^{\alpha}})j}(z_{l}).{M}^{l}\nabla_{x}u^{(\frac{1}{\epsilon^{\alpha}})j}(z_{l})
+\Theta(\epsilon^{(\frac{5}{2}-\alpha)}),
\end{equation}
where ${M}^{l}$ is the polarisation tensor associated to $B^{l}$ and $\overline{\lambda_{\epsilon}^{E(\frac{1}{\epsilon^{\alpha}})}}$ is the harmonic average of the $(\lambda_{\epsilon}^{E(\frac{1}{\epsilon^{\alpha}})j}).$
\end{thm}
\begin{proof}
We have
\begin{eqnarray*}
\frac{1}{\lambda^{i}}-\frac{1}{m}\sum_{j=1}^{m}\frac{1}{\lambda^{ij}_{\epsilon}}&=& \frac{1}{m}\sum_{j=1}^{m}<\frac{1}{\lambda^{i}}u^{ij}-\upsilon_{\epsilon}^{ij},u^{ij}>\\
                                                        \nonumber\\
  &=&\epsilon^{2}\frac{1-k}{\lambda^{{i}^{2}}}
(|B||\nabla_{x}u^{ij}(0)|^{2}\\
                                           \nonumber\\
 & +&\lambda^{i}\int_{\partial B}\frac{\partial \upsilon}{\partial\nu_{\xi}}|_{-}(\xi)\frac{\nabla_{x}u^{ij}}{\lambda^{i}}(0).\xi d\sigma_{\xi})+\theta(\frac{\epsilon^{\frac{5}{2}}}{\sqrt{\lambda^{i}}}),
\end{eqnarray*}
and using the assertion (32), we get
$$\frac{1}{\lambda^{i}}-\frac{1}{m}\sum_{j=1}^{m}\frac{1}{\lambda^{ij}_{\epsilon}}=\frac{\epsilon^{2}}{m \lambda^{i}}\sum_{j=1}^{m}\nabla_{x}\upsilon^{ij}(0).{M^{l}}\nabla_{x}\upsilon^{ij}(0)+\theta(\frac{\epsilon^{\frac{5}{2}}}{\sqrt{\lambda^{i}}}).$$
Then, inserting this in Theorem 3.6 in connection with $v^{ij}=\frac{u^{ij}}{\sqrt{\lambda^{i}}}$, we prove that
$$\overline{\lambda _{\epsilon}^{i}}-\lambda^{i}=
\epsilon^{2}\sum_{j=1}^{m}\sum_{l=1}^{N}\nabla_{x}u^{(\frac{1}{\epsilon^{\alpha}})j}(z_{l}).{M}^{l}\nabla_{x}u^{(\frac{1}{\epsilon^{\alpha}})j}(z_{l})
+\Theta(\epsilon^{(\frac{5}{2}-\alpha)}),$$
where ${M}^{l}$ is the polarisation tensor associated to $B^{l}.$
Finally, we take $i=E(\frac{1}{\epsilon^{\alpha}})$, for $0\leq\alpha\leq1$, we get our main result.
\end{proof}

\end{document}